\documentclass[12pt, reqno]{amsart}
\usepackage{amsmath, amsthm, amscd, amsfonts, amssymb, graphicx, color}
\usepackage[bookmarksnumbered, colorlinks, plainpages]{hyperref}
\hypersetup{colorlinks=true,linkcolor=red, anchorcolor=green, citecolor=cyan, urlcolor=red, filecolor=magenta, pdftoolbar=true}
\usepackage{comment}

\newtheorem{theorem}{Theorem}[section]
\newtheorem{lemma}[theorem]{Lemma}

\newtheorem{corollary}[theorem]{Corollary}
\theoremstyle{definition}
\newtheorem{definition}[theorem]{Definition}
\newtheorem{example}[theorem]{Example}

\theoremstyle{remark}

\numberwithin{equation}{section}

\newcommand{\Z}{\mathbb{Z}}
\newcommand{\R}{\mathbb{R}}
\newcommand{\N}{\mathbb{N}}
\newcommand{\Comp}{\mathbb{C}}

\newcommand{\supp}{\textrm{supp}~}

\newcommand{\Sc}{\mathcal{S}(\mathbb{R}^{n})}
\newcommand{\Sl}{\mathcal{S}'(\mathbb{R}^{n})}
\newcommand{\Dl}{\mathcal{D}'}

\newcommand{\ssupp}{\textrm{sing supp}~}

\begin{document}

\setcounter{page}{1}

\title[Spectral theory on localized Sobolev]{Spectrum of differential operators with elliptic adjoint on a scale of localized Sobolev spaces}

\author[L.M. Salge \MakeLowercase{and} E.R. Arag\~ao-Costa]{Lu\'is M\'arcio Salge$^1$ \MakeLowercase{and} \'Eder R\'itis Arag\~ao Costa$^1$}

\address{$^{1}$Instituto de Ci\^encias Matem\'aticas e de Computa\c c\~ao, University of S\~ao Paulo, CEP: 13566-590, S\~ao Carlos - SP, Brasil.}
\email{\textcolor[rgb]{0.00,0.00,0.84}{lmsalge@icmc.usp.br;
ritis@icmc.usp.br}}



\let\thefootnote\relax\footnote{Copyright 2016 by the Tusi Mathematical Research Group.}



\date{Received: xxxxxx; Revised: yyyyyy; Accepted: zzzzzz.
\newline \indent $^{*}$Corresponding author}

\begin{abstract} 
In this paper we provide a complete study of the spectrum of a constant coefficients differential operator on a scale of localized Sobolev spaces, $H^{s}_{loc}(I)$, which are Fr\'echet spaces. This is quite different from what we find in the literature, where all the relevant results are concerned with spectrum on Banach spaces. 

Our aim is to understand the behavior of all the three types of spectrum (point, residual and continuous) and the relation between them and those of the dual operator. The main result we present shows that there is no complex number in the resolvent set of such operators, which suggest a new way to define spectrum if we want to reproduce the classical theorems of the Spectral Theory in Fr\'echet spaces.
\end{abstract} 
\maketitle





\section{Introduction and preliminaries}


In this work we present a complete study about the spectrum of a constant coefficients differential operator of order $m\in \N$, $a(D)$, whose adjoint $a(D)\mbox{*}$ is elliptic, seen as a pseudo-differential operator on a interval $I\subset \R$, that is, seen as
$$
a(D): H^{s+m}_{0}(I) \subset H^{s}_{loc}(I) \longrightarrow H^{s}_{loc}(I),\, s\in \R.
$$
Here, $H^{s}_{loc}(I)$ is endowed with the topology generated by a family of seminorms $\left(p^{(s)}_{j}\right)_{j \in \N}$ given by $p^{(s)}_{j}(f) \doteq \displaystyle \left\|\varphi_j f\right\|_{H^{s}(\R)},\, f \in H^{s}_{loc}(I),$
where, for each $j \in \N, \, I_{j} = (a_{j},b_{j})$ is such that $[a_{j},b_{j}] \subset (a_{j+1},b_{j+1}),$ with $I = \bigcup_{j \in \N}[a_{j},b_{j}],$ and $\varphi_j \in C^{\infty}_{c}(I_{j+1})$ satisfies $\varphi_j =1$ in $[a_{j},b_{j}]$. 

When we indicate $a(D)$ as above, we mean that in $H^{s+m}_{0}(I)$ we consider the topology induced from $H^{s}_{loc}(I).$

This study was developed inspired by what happen with the Laplace operator on $L^2(I)$. Here, we replace $L^2(I)$ by $H^{s}_{loc}(I)$ and $H^1_{0}(I) \cap H^2(I)$ by $H^{m+s}_{0}(I)$, as suggested by the definitions we found in \cite{henry}.

The best conclusions we obtain are when we consider the Laplace operator on an interval $I$ as
$$
\Delta: H^{2}_{0}(I) \subset L^{2}_{loc}(I) \longrightarrow L^{2}_{loc}(I).
$$
For it, we calculate its closure and compare its spectrum in three stages:
\begin{itemize}
\item[(1)] When it is defined on $H^{2}_{0}(I).$ 

\item[(2)] When its domain is $H^{1}_{0}(I) \cap H^{2}(I)$, where we call it $\Delta_{L^2}$; and 

\item[(3)] When it is defined on $H^{2}_{loc}(I).$ This, as we are going to see, is the domain of the closure $\overline{\Delta}$.
\end{itemize}

In particular, we prove that $\sigma_c(\Delta)=\sigma_r(\Delta^{*})=\sigma_p(\overline{\Delta})=\Comp$ and $\sigma_c(\Delta_{L^2})=\Comp \setminus \left\{-\frac{\pi^2n^{2}}{l(I)^2}: n \in \N \right\}$, where $l(I)$ is the length of $I$.



\subsection{Preliminary concepts and results}\label{sec2}

In this section we present some definitions and results from Functional Analysis which were the basic tools to make this work possible.

We begin by defining Fr\'echet spaces and a consequence of Hahn-Banach Theorem for Fr\'echet spaces.

\begin{definition}
Let $X$ be a topological vector space. $X$ is said to be a Fr\'echet space if it is Hausdorff, complete and its topology is given by a countable family of seminorms.
\end{definition}

Some examples of Fr\'echet space are the space of smooth functions $C^{\infty}(\Omega),$ where $\Omega \subset \R^{n},$ the Schwartz space $\Sc$ and the localized Sobolev spaces $H^{s}_{loc}(\Omega)$ for $s \in \R.$


Usually, the dual of a Fr\'echet space is equipped with the weak* topology which is also generated by a family of semi-norms, we will use it in this work and we will explain this latter on. 

We have $[C^{\infty}_{c}(\Omega)]' = \Dl(\Omega)$ is the space of distributions.
The spaces $\Sl = [\Sc]',$ the space of tempered distributions, and $\mathcal{E}'(\Omega) = [C^{\infty}(\Omega)]',$ the space of distributions with compact support. 

Finally we define the formal transpose of an operator, which is also used in this work.
Given $L: C^{\infty}_{c}(\Omega) \longrightarrow C^{\infty}_{c}(\Omega)$ a continuous linear operator, if there exists $L':C^{\infty}_{c}(\Omega) \longrightarrow C^{\infty}_{c}(\Omega)$, a linear and continuous operator, such that 
$$
\int_{\Omega} L(\psi)(x) \phi(x) dx = \int_{\Omega} \psi(x) L'(\phi)(x) dx,
$$
i.e., $\langle L\psi, \phi \rangle = \langle \psi, L'\phi \rangle,$
for every $\psi, \phi \in C^{\infty}_{c}(\Omega),$ then $L'$ is called the formal transpose of $L$ and vice-versa.

\begin{theorem} 
\label{hahn-banach-frechet}
Let $(X,(p_{j})_{j \in \N})$ be a Fr\'echet space. If $M \subset X$ is a subspace such that $\overline{M} \neq X,$ then there exists a non-null $G \in X'$ that satisfies 
$
\langle G, x \rangle = 0, \; \forall \; x \in M.
$
\end{theorem}

We recall the Fourier Transform which is defined by
$$
(\mathcal{F}\psi) (\xi) \doteq \int_{\R^{n}} e^{-2\pi i x \xi} \psi(x) dx\;, \psi \in L^{1}(\R^{n}) \text{ and } \xi \in \R^n.
$$
 $\hat{\psi}$ is also used to denote the Fourier Transform of $\psi.$

\begin{definition}
Let $\Omega \subset \R^{n}$ a given interval, the Sobolev space $H^{1}(\Omega)$ is the following set   
\begin{equation*}
     \left\{u \in L^{2}(\Omega); \exists \; g_{\alpha} \in L^{2}(\Omega); \displaystyle \int_{\Omega} u \partial^{\alpha}\phi = (-1)^{|\alpha|} \int_{\Omega} g_{\alpha} \phi, \mbox{for each } |\alpha|= 1, \phi \in C^{\infty}(\Omega) \right\}.
\end{equation*}
Here, the functions $g_{\alpha}$ are denoted by $\partial^{\alpha}u$ and each $\partial^{\alpha}u$ is said to be the weak $\alpha-$derivative of $u.$ Moreover, the usual topology of $H^{1}(\Omega)$ is determined by the following norm
$
    \|u\|_{H^{1}(\Omega)} \doteq \sum_{0 \leq |\alpha| \leq 1}  \|\partial^{\alpha}u\|_{L^{2}(\Omega)}.
$

Given a natural number $m \geq 2,$ the Sobolev space $H^{m}(\Omega)$ is defined as 
$$
    H^{m}(\Omega) = \{u \in H^{m-1}(\Omega); \partial^{\alpha}u \in H^{m-1}(\Omega) \mbox{ for each } |\alpha| = 1 \}
$$
and its usual topology is defined by the norm 
$
    \|u\|_{H^{m}(\Omega)} \doteq \sum_{0 \leq |\alpha| \leq m} \|\partial^{\alpha}u\|_{L^{2}(\Omega)}.
$
\end{definition}

\begin{definition}
Let $m \in \N,$ $H^{m}_{0}(\Omega)$ is the closure of $C^{\infty}_{c}(\Omega)$ in $H^{m}(\Omega)$ with the induced topology.
\end{definition}

The following theorem, which can be seen in \cite{brezisaf}, gives an alternative way to describe the space $H^{m}(\R^{n})$ by using the Fourier transform.

\begin{theorem} \label{teosob}
For $m \in \N$ we have 
\begin{equation*}
    H^{m}(\R^{n}) = \big\{u \in \Sl; (1+|\xi|^{2})^{m/2}\hat{u} \in L^{2}(\R^{n}) \big\}.
\end{equation*}
Furthermore, the norm given by
$
    \|u\|_{m} \doteq \left\|(1+|\xi|^{2})^{m/2}\hat{u}\right\|_{L^{2}(\R^{n})}
$ is equivalent to $\|\cdot\|_{H^{m}(\R^{n})}.$
\end{theorem}

This result suggests a way to define Sobolev spaces for any $s \in \R.$ Given $s \in \R$ we define 
$$
H^{s}(\R^{n}) = \left\{u \in \Sl; (1+|\xi|^{2})^{s/2}\hat{u} \in L^{2}(\R^{n}) \right\}.
$$

\begin{definition}
Given an open set $\Omega \subset \R^{n}$ and $s \in \R$ we define the local Sobolev space of order $s$ on $\Omega$ as
$$
    H^{s}_{loc}(\Omega) =\big \{u \in \Dl(\Omega); \phi u \in H^{s}(\Omega), \; \forall \; \phi \in C^{\infty}_{c}(\Omega)\big \}.
$$
\end{definition}

For each $s\in \R$, $H^{s}_{loc}(\Omega)$ is a Fr\'echet space and its semi-norms are given by 
$$
p_{j} \doteq \|\phi_{j} u\|_{H^{s}(\R^{n})},
$$
where $\phi_{j} = 1$ in $\Omega_{j},$ $\phi_{j} \in C^{\infty}_{c}(\Omega_{j+1})$ and $\Omega_{j} \subset \Omega$ is a sequence of open sets that exhaust $\Omega.$ 

When $s=0$, $H^{0}_{loc}(I) = L^{2}_{loc}(I)$ which, as we will see later on, is used as a 'base space' for the Laplacian.


\bigskip
\begin{definition}[Semiglobal symbol of a pseudo-differential operator of order $m$]
Let  $a \in C^{\infty}(\Omega \times \R^{n})$ and $m \in \R$ be such that for every compact $K \subset \Omega$ and multi-indexes $\alpha, \beta$ there exists $C_{K,\alpha, \beta}>0$ with 
$\left|\partial_{x}^{\beta} \partial_{\xi}^{\beta} a(x,\xi)\right| \leq C_{K,\alpha, \beta} \left(1+|\xi|\right)^{m - |\beta|},$ for each $\xi \in \R^{m}, x \in K.$

The function $a$ is said to be a (Semiglobal) symbol of a pseudo-differential operator of order $m$ and the class of all (Semiglobal) symbols of order $m$ is denoted by $S^{m}(\Omega).$
\end{definition}

Now we present the definition of operator of order $m$ we found in \cite{henry}, which was as an inspiration to this work. By means of this definition was possible to build a link between the spectrum of $\Delta: H^{1}_{0}(I) \cap H^{2}(I) \subset L^{2}(I) \to L^{2}(I)$ and the spectrum of $\Delta$ defined on a localized Sobolev space.

\begin{definition}
Given $m \in \R,$ a linear operator 
$A: C^{\infty}_{c}(\Omega) \to C^{\infty}(\Omega)$
is said to be a operator of order $m$ if, for every $s \in \R,$ $A$ extends to a linear operator
\begin{equation*}
    A_{s}: H_{0}^{s+m}(\Omega) \subset H^{s+m}_{loc}(\Omega) \to H^{s}_{loc}(\Omega).
\end{equation*}
\end{definition}

The proof of the next theorem can be found in \cite{henry}
\begin{theorem}
If $p \in S^{m}(\Omega),$ then $p(x,D)$ is an operator of order $m.$ 
\end{theorem}

It is well known that many of differential operators which are studied in PDE are not continuous and, in some cases, not even closed, so the concept of closed and closable operators are fundamental. In the last section of this paper the operators are just closable so, at this point, we present the definitions and some basic results about closed, closable operators and its spectrum.

\begin{definition} 
Consider a Fr\'echet space $X$ and a linear operator 
$A: D(A) \subset X \to X.$
The graph of $A$ is the set  
\begin{equation*}
    G(A) = \{(u, Au): u \in D(A)\} \subset X \times X.
\end{equation*}
The operator $A$ is said to be a closed operator if its graph $G(A) \subset X \times X$ is a closed set.
\end{definition}

\begin{definition} 
Consider a Fr\'echet space $X$ and a linear operator 
$A: D(A) \subset X \to X.$
We say that $A$ is a closable, if there exists a closed linear operator 
$\overline{A}: D(\overline{A}) \subset X \to X,$
with $D(A) \subset D(\overline{A})$ and $Au = \overline{A}u$, for each $u \in D(A)$.
\end{definition}

\begin{definition}\label{def-espectro} 
Let $X$ be a complex Fr\'echet space and $A:D(A)\subset X\longrightarrow X$ be a linear operator.
The resolvent set of $A$, denoted by $\rho (A)$, is the set of all $\lambda \in {\mathbb C}$ such that:
\begin{itemize}
\item[(a)] The operator $\lambda - A :D(A)\subset X\longrightarrow X$ is injective.
\item[(b)] The range of  $\lambda - A :D(A)\subset X\longrightarrow X$ is dense in $X$.
\item[(c)] The inverse $(\lambda - A)^{-1} :R(\lambda - A)\subset X\longrightarrow X$ is continuous.
\end{itemize}

If $\lambda \in \rho(A)$, the operator $(\lambda - A)^{-1} :R(\lambda - A)\subset X\longrightarrow X$ is called the resolvent of $A$ on $\lambda$.

Finally, we define the spectrum of $A$, indicated by $\sigma (A)$, as $\sigma(A)={\Comp}\setminus \rho(A).$

Next we define, for a closed operator $A$, respectively the point spectrum, residual spectrum and continuous spectrum as follows: 
\begin{itemize}
    \item[(a)] \textbf{Point Spectrum:}
    
    $\sigma_{p}(A)$ $\doteq \big \{\lambda \in \Comp; \lambda - A \; \mbox{is not injective} \big \},$
    \item[(b)] \textbf{Residual Spectrum:}
    
    $\sigma_{r}(A) \doteq \big \{\lambda \in \Comp; \lambda - A \; \mbox{is injective with} \; \overline{R(\lambda - A)} \neq X \big\}, $ 
    \item[(c)] \textbf{Continuous Spectrum:}
    
    $\sigma_{c}(A) \doteq \big \{\lambda \in \Comp; \lambda - A  \mbox{ is injective, } \overline{R(\lambda - A)} = X \mbox{ but }$ \\ $(\lambda - A)^{-1}: R(\lambda - A) \to X \mbox{ is not continuous} \big \}.$
\end{itemize}
Note that $\sigma(A) = \sigma_{p}(A) \cup \sigma_{r}(A) \cup \sigma_{r}(A).$
\end{definition}

\begin{lemma} Let $X$ be a Fr\'echet space. If $A:D(A)\subset X \longrightarrow X$ is a closed operator, then
$
\rho(A)=\big \{ \lambda \in \Comp: \lambda-A:D(A)\longrightarrow X \text{ is bijective}\}
$
\end{lemma}

The next result allows us to study the spectrum of a closable operator $A$ by means of the spectrum of its closure $\overline{A}$. 

\begin{theorem} Consider $X$ a Fr\'echet space. If $A:D(A)\subset X \longrightarrow X$ is closable and $\overline{A}:D(\overline{A})\subset X \longrightarrow X$ is its closure, then
$
\sigma(A)=\sigma(\overline{A}).
$ \label{fechadofrechet}
\end{theorem}

\begin{proof} To show that $\sigma(\overline{A})=\sigma(A)$ is the same as to prove that  $\rho(\overline{A})=\rho(A)$. 

So, fix $\lambda \in \rho(\overline{A})$, since $D(A) \subset D(\overline{A})$ we have that $\lambda - \overline{A}\mid_{D(A)} = \lambda - A$ consequently $\lambda - A$ is injective.

If $f \in X = R(\lambda - \overline{A})$ then there exists $u \in D(\overline{A})$ with $(\lambda - \overline{A}) u = f.$
By the definition of domain and range of $\overline{A}$ there exists a sequence $(u_{j})_{j \in \N} \in D(A)$ with $u_{j} \to u$ and $(\lambda - A)u_{j} \to f,$ hence $f \in \overline{R(\lambda - A)}.$

It remains to prove that $(\lambda - A)^{-1} : R((\lambda - A))\subset X \to X$ is a continuous operator. To do so, consider a sequence $(f_{j})_{j \in \N} \subset R(\lambda - A)$ such that 
$f_{j} \to f \in  R(\lambda - A),$ then there is $u \in D(A)$ with $f=(\lambda - A)u.$ Now we just have to show that $u_{j} \to u.$

Note that $f_{j} = (\lambda - A)u_{j} = (\lambda - \overline{A})u_{j},$ $u_{j} \in D(A) \subset D(\overline{A}),$ or equivalently, $u_{j} = (\lambda - \overline{A})^{-1}f_{j}$ and since $(\lambda - \overline{A})^{-1}$ is continuous we have
$$u = (\lambda - \overline{A})^{-1}f = \lim_{j \to \infty} (\lambda - \overline{A})^{-1}f_{j} = \lim_{j \to \infty}u_{j}, \text{ hence } \rho(\overline{A}) \subset \rho(A).$$

Conversely, consider $\lambda \in \rho(A)$ so $\lambda - A$ is injective, $X = \overline{R(\lambda - A)}$ and 
$(\lambda - A)^{-1}: R(\lambda - A) \subset X \to X$
is continuous. 

Let us prove that $\lambda - \overline{A}$ is bijective. Indeed, 
if $u \in D(\overline{A})$ with $(\lambda - \overline{A})u =0$ then there is a sequence $(u_{j})_{j \in \N} \subset D(A)$ with $u_{j} \to u$ and $f_{j} \doteq (\lambda - A)u_{j} \to 0.$

Note that $0 \in R(\lambda - A),$ therefore from the continuity of $(\lambda - A)^{-1}$ we have  
$$u_{j} = (\lambda - A)^{-1}f_{j} \to (\lambda - A)^{-1} 0 = 0,$$
in other words, $u=0$ and it follows that $\lambda - \overline{A}$ is injective.

Now, let $f \in X = \overline{R(\lambda - A)},$ so there is  $(u_{j})_{j \in \N} \subset D(A)$ with $f_{j} = (\lambda - A)u_{j}$ and $f_{j} \to f.$

Hence $(f_{j})_{j}$ is a Cauchy sequence, i.e., 
$f_{j}-f_{l} = (\lambda - A)(u_{j}-u_{l}) \to 0$
and, since $(\lambda - A)^{-1}$ is continuous,
$u_{j}-u_{l} = (\lambda - A)^{-1}(f_{j}-f_{l}) \to 0.$ 
We conclude that $(u_{j})_{j \in \N}$ is a Cauchy sequence, so there is $u \in X$ such that $u_{j} \to u.$
Therefore $(u,f) \in \overline{G(\lambda - \overline{A})} = G(\lambda - \overline{A}),$ then $u \in D(\overline{A})$ and $f = (\lambda - \overline{A})u \in R(\lambda - \overline{A}).$
\end{proof}

\section{Main results}

\subsection{Spectrum of differential operators with elliptic dual}

Here we present the main results of this paper which were achieved through the study of the spectrum of differential operators with constant coefficients with elliptic dual.

The first result is more general, and works for differential operators with constant coefficients with hypoelliptic dual, but to give a more precise description we need to restrict a bit more the class of operators to those with elliptic dual. At the end we apply the result to the Laplacian. 

Consider a symbol $a \in S^{m}(\R)$ given by $a(\xi)=\sum_{k=0}^{m}a_k\xi^{k}$, $m \in \N,$ $a_k \in {\mathbb C}$, and the differential operator $a(D)=\sum_{k=0}^{m}(2\pi i)^{-k}a_k\frac{d^k}{dx^k}$, determined by it, defined on the following scales
$$
a(D): H^{s+m}_{0}(I) \subset H^{s}_{loc}(I) \to H^{s}_{loc}(I), \; s \in \R.
$$
Our goal is to compare its spectrum with that from its dual
$$
a(D)\mbox{*}: D(a(D)\mbox{*}) \subset H^{-s}_{c}(I) \to H^{-s}_{c}(I)
$$
where
$$
D(a(D)\mbox{*}) \doteq \left\{g \in H^{-s}_{c}(I); g \circ a(D): H^{s+m}_{loc}(I) \subset H^{s}_{loc}(I) \to \Comp \; \mbox{ is continuous} \right\}
$$
and it satisfies the relation
$$
\left\langle u, a(D)\mbox{*}\psi \right\rangle = \left\langle a(D)u, \psi \right\rangle = \left\langle \displaystyle \sum_{j=0}^{m}(2\pi i)^{-j} a_{j} \frac{d^{j}u}{dx^{j}}, \psi \right\rangle = \left\langle u, \sum_{j=0}^{m} (-2\pi i)^{-j} a_{\alpha} \frac{d^{j}\psi}{dx^{j}} \right\rangle
$$
for $u \in H^{s+m}_{loc}(I)$ and $\psi \in C^{\infty}_{c}(I).$

Before stating the next theorem we need some definitions, Theorem $6.36$ from \cite[pg - 216]{follandedp} and the following theorem.

\begin{theorem} \label{multi-sobolew}
For each $s \in \R$ and $\phi \in \Sc,$ the map $M_{\phi}: H^{s}(\R^{n}) \longrightarrow H^{s}(\R^{n})$, given by
$M_{\phi}(u) = \phi u,$ is linear and continuous.
\end{theorem}

\begin{definition}
Given $\Omega \subset \R^{n}$ an open set and $a(D): {\mathcal D}'(\Omega) \to {\mathcal D}'(\Omega)$ a differential operator with constant coefficients, we say that $a(D)$ is hypoelliptic if, for any $u \in {\mathcal D}'(\Omega),$ we have 
$$
\ssupp [a(D)u] = \ssupp u.
$$
\end{definition}

\begin{definition} \label{def-elipt}
Given $\Omega \subset \R^{n}$ an open set and $a(x,D): \Dl(\Omega) \to \Dl(\Omega)$ a differential operator of order $m,$ we say that $a(x,D)$ is elliptic if, for any compact set $K  \subset \Omega,$ there exists positive constants $c_{K}, C_{K}$ such that 
$$
|a(x,\xi)| \geq c_{K}|\xi|^{m} \mbox{ for any } x \in K \mbox{ and } |\xi| \geq C_{K}. 
$$ 
\end{definition}

\begin{theorem}[H\"ormander] \label{hormander}
Let $a(\xi) = \sum_{|\alpha| \leq m} a_{\alpha}\xi^{\alpha}$ be a symbol of a differential operator, $a(D),$ of order $m>0.$ 
The following statements are equivalent:
\begin{enumerate}
    \item If $|\zeta| \to \infty$ for $\mathcal{Z}(a),$ then $|\Im \zeta| \to \infty;$
    \item If $|\xi| \to \infty$ in $\R^{n},$ then $d_{P}(\xi) \to \infty;$
    \item There exists $\delta, C,R>0$ such that $d_{P}(\xi) \geq C|\xi|^{\delta}$, if $|\xi| > R$ in $\R^{n};$
    \item There exists $\delta, C,R>0$ such that $|a^{(\alpha)}(\xi)| \leq C |\xi|^{-\delta |\alpha|}|a(\xi)|,$ for all $\alpha$ and $\xi \in \R^{n}$ with $|\xi| > R;$
    \item There exist $\delta>0$ such that if $f \in H^{s}_{loc}(\Omega),$ where $\Omega \subset \R^{n}$ is an open set, then every solution $u$, of $a(D)u = f$, belongs to $H^{s+\delta m}_{loc}(\Omega);$
    \item $a(D)$ is hypoelliptic.
\end{enumerate}
\end{theorem}


The domain of $a(D)\mbox{*}$ is a subset of $[H^{s}_{loc}(I)]',$ so first we present a theorem which gives a characterization for $[H^{s}_{loc}(I)]'$ and then, as one of the results of our work, we localize the domain $D[a(D)\mbox{*}].$

\begin{theorem}
For each $s \in \R,$ it holds that $[H^{s}_{loc}(I)]' = H^{-s}_{c}(I)$ and $[H^{-s}_{c}(I)]' = H^{s}_{loc}(I),$ where $[H^{s}_{loc}(I)]'$ indicates the dual space of $H^{s}_{loc}(I)$, $[H^{-s}_{c}(I)]'$ the dual space of $H^{-s}_{c}(I)$ and the equalities are in the sense that there exists a $T: H^{-s}_{c}(I) \rightarrow [H^{s}_{loc}(I)]'$ linear continuous bijection.
\end{theorem}

\begin{theorem}
Let $a(D):H^{s+m}_{0}(I) \subset H^{s}_{loc}(I) \longrightarrow H^{s}_{loc}(I)$ be a differential operator of order $m$ with hypoelliptic formal transpose $a(D)'.$ 
There exists $0<\delta \leq 1$ such that
$
H^{-s+m}_{c}(I) \subset D\left[a(D)\mbox{*}\right] \subset H^{-s+\delta m}_{c}(I).
$
\end{theorem}

\begin{proof}
\textbf{Part I) $H_{c}^{-s+m}(I) \subset D\left[a(D)\mbox{*}\right]$}

First of all, consider $u \in H^{s+m}_{loc}(I)$ and $g \in H^{-s+m}_{c}(I).$ Denote $K=\supp g,$ so there  exists $j \in \N$ with $\varphi_{j} \in C^{\infty}_{c}(I)$ and $\varphi_{j} \equiv 1$ in a neighborhood of $K,$ where $\varphi_{j}$ is a test function from the family of seminorms of $H^{s}_{loc}(I).$ 

In such conditions, as $\frac{d^{k}g}{dx^{k}} \in H^{-s}(\R)$, for $1 \leq k \leq m,$ we have $\left| \left\langle g, a(D)u \right\rangle \right|=\left| \left\langle g,  \sum_{k=0}^{m} (2 \pi i)^{-k}a_{k}\frac{d^{k}u}{dx^{k}} \right\rangle \right| = \left| \left\langle  \sum_{k=0}^{m} (-2 \pi i)^{-k}a_{k}\frac{d^{k} g}{dx^{k}}, u \right\rangle \right| =$
\begin{eqnarray*}
         &=& \left| \left\langle  \sum_{k=0}^{m} (-2 \pi i)^{-k}a_{k}\frac{d^{k} g}{dx^{k}}, \varphi_{j}u \right\rangle \right| \leq \sum_{k=0}^{m} A_{k} \left\|\frac{d^{k}g}{dx^{k}}\right\|_{H^{-s}(\R)} \|\varphi_{j}u\|_{H^{s}(\R)} \leq \\
    &\leq&
    \sum_{k=0}^{m} \tilde{A}_{k} \left\|g\right\|_{H^{-s+k}(\R)} \|\varphi_{j}u\|_{H^{s}(\R)} \leq  C\|g\|_{H^{-s+m}(\R)} p_{j}^{(s)}(u),
\end{eqnarray*}
because $H^{-s+m}(\R) \hookrightarrow H^{-s+k}(\R) \hookrightarrow H^{-s}(\R)$, for each $1 \leq k \leq m,$ where $A_{0},  \ldots, A_{m},$ $\tilde{A}_{0}, \ldots, \tilde{A}_{m}$ and $C$ are constants.

Note that it was possible to obtain the continuity relative to the topology of $H^{s}_{loc}(I)$ only because $\supp \left[\sum_{k=0}^{m} (-2 \pi i)^{-k}a_{k}\frac{d^{k} g}{dx^{k}} \right] \subset \supp g \subset K$ and $\frac{d^{k}g}{dx^{k}} \in H^{-s}(\R)$ for $0 \leq k \leq m.$

Therefore $g \in D[a(D\mbox{*})]$ and
$
a(D)\mbox{*}g = \sum_{k=0}^{m} (-2 \pi i)^{-k}a_{k}\frac{d^{k}g}{dx^{k}}.
$
Also, it follows that $H^{-s+m}_{c}(I) \subset D\left[a(D)\mbox{*}\right].$

\textbf{Part II) $D\left[a(D)\mbox{*}\right] \subset H^{-s+\delta m}_{c}(I)$}

Now the goal is to show that there is a $0 < \delta \leq 1$ such that $D\left[a(D)\mbox{*}\right] \subset H^{-s+\delta m}_{c}(I).$ 

Given $g \in D\left[a(D)\mbox{*}\right]$, by definition of the domain $a(D)\mbox{*},$ there are $M>0$ and a seminorm $p^{(s)}_{j}(\cdot)$ such that 
$
\left|\left\langle g, a(D)u \right\rangle\right| \leq M p^{(s)}_{j}( u),\; u \in H^{s+m}_{0}(I).
$
Observe that $a(D)\mbox{*}g \in \left[H^{s}_{loc}(I)\right]' = H^{-s}_{c}(I)$ is the continuous extention of
$
g \circ a(D): H^{s+m}_{0}(I) \subset H^{s}_{loc}(I) \longrightarrow \Comp.
$

Note that, for $\psi \in C^{\infty}_{c}(I),$ it holds  
$$
\left\langle a(D)\mbox{*}g, \psi \right\rangle = \left\langle g, a(D)\psi \right\rangle = \left\langle g,\sum_{l=0}^{m}(2 \pi i)^{-l}a_{l}\frac{d^{l}\psi}{dx^{l}} \right\rangle = \left\langle \sum_{l=0}^{m}(-2 \pi i)^{-l}a_{l}\frac{d^{l}g}{dx^{l}}, \psi, \right\rangle 
$$
which means $\sum_{l=0}^{m}(-1)^{l}a_{l}\frac{d^{l}g}{dx^{l}} = a(D)\mbox{*}g$ as distributions in $\Dl(I)$ and, since $a(D)\mbox{*}g \in H^{-s}_{c}(I) \hookrightarrow H^{-s}_{loc}(I),$ it implies that $\sum_{l=0}^{m}(-2 \pi i)^{-l}a_{l}\frac{d^{l}g}{dx^{l}} \in H^{-s}_{loc}(I).$

Since $\sum_{l=0}^{m}(-2 \pi i)^{-l}a_{l} \frac{d^{l}}{dx^{l}}$ is hypoelliptic, by Theorem \ref{hormander}, there is $0 < \delta \leq 1$ such that $g \in H^{-s+\delta m}_{loc}(I).$
Henceforth, $g \in H^{-s+\delta m}_{loc}(I)$ with $\supp g$ compact, i.e., $g \in H^{-s+\delta m}_{c}(I),$ and that completes the proof.

\end{proof}

\begin{corollary}
In the theorem above, if $a(D)'$ is elliptic, then $\delta = 1$  and, consequently, $D\left[a(D)\mbox{*}\right] = H^{-s+m}_{c}(I).$ Furthermore, it holds that
$$
a(D)\mbox{*}g = \sum_{k=0}^{m} (-2 \pi i)^{-k}a_{k}\frac{d^{k}g}{dx^{k}}, \mbox{ para } g \in H^{-s+m}_{c}(I).
$$
\end{corollary}

Now we compare the sets $\sigma(a(D))$ and $\sigma(a(D)\mbox{*}),$ where 
$
a(D): H^{s+m}_{0}(I) \subset H^{s}_{loc}(I) \to H^{s}_{loc}(I) 
$
is a differential operator with constant coefficients and 
$
a(D)\mbox{*}: D(a(D\mbox{*})) \subset H^{-s}_{c}(I) \to H^{-s}_{c}(I)
$
is its hypoelliptic dual.


\begin{theorem}
Let $a(D):H^{s+m}_{0}(I) \subset H^{s}_{loc}(I) \longrightarrow H^{s}_{loc}(I)$ be a linear differential operator with constant coefficients of order $m$ such its formal transpose $a(D)'$ is hypoellitic. Under such conditions, the following inclusions are true
\bigskip
\begin{itemize} 
\item[(i)]
$\sigma_{p}(a(D))\cup \sigma_{r}(a(D)) \subset \sigma_{p}(a(D)\mbox{*})\cup \sigma_{r}(a(D)\mbox{*})$.

\item[(ii)] 
$\sigma_{p}(a(D)\mbox{*}) \subset \sigma_{p}(a(D)) \cup \sigma_{r}(a(D))$; and

\item[(iii)]
$\sigma_{r}(a(D)\mbox{*}) \subset \sigma_{p}(a(D)) \cup \sigma_{c}(a(D)).$
\end{itemize}
\end{theorem}

\begin{proof}
The proof is splitted into four steps.

\textbf{Step I: $\sigma_{r}(a(D)) \subset \sigma_{p}\left(a(D)\mbox{*}\right)$.}

Given $\lambda \in \sigma_{r}(a(D)),$ by the definition of residual spectrum, $\lambda - a(D)$ is injective and $\overline{R(\lambda - a(D))} \neq H^{s}_{loc}(I)$ so, by Theorem \ref{hahn-banach-frechet}, there is a functional $g \neq 0$, which is an element of $H^{-s}_{c}(I),$ that satisfies 
$
\left\langle g, (\lambda - a(D))u \right\rangle = 0,$ for each $u \in H^{s+m}_{0}(I).
$

From the above equality, it follows that $g \in D\left[a(D)\mbox{*}\right] \subset H^{-s+\delta m}_{c}(I)$ with 
$$
\left\langle \left(\lambda - a(D)\mbox{*}\right)g, u \right\rangle = \left\langle g, \left(\lambda - a(D)\right)u \right\rangle = 0, \; \forall \; u \in H^{s+m}_{0}(I),
$$
in other words $\left(\lambda - a(D)\mbox{*}\right)g = 0$ with $g \neq 0$ therefore $\lambda \in \sigma_{p}(a(D)\mbox{*}).$

\textbf{Step II: $\sigma_{p}(a(D)) \subset \sigma_{p}(a(D)\mbox{*}) \cup \sigma_{r}(a(D)\mbox{*}).$}

If $\lambda \in \sigma_{p}(a(D)),$ for some $u\not =0$ in $H^{s+m}_{0}(I),$ we have $(\lambda - a(D))u = 0.$ So, for any $g \in H^{-s}_{c}(I)$ it is true that $\langle g, (\lambda - a(D))u \rangle = 0$ and then 
$$
\left\langle (\lambda - a(D)\mbox{*})g, u \right\rangle = \left\langle g, (\lambda - a(D))u \right\rangle = 0, \; \forall \; g \in D(a(D)\mbox{*})\subset H^{-s}_{c}(I).
$$


On the other hand, if $\lambda - a(D)\mbox{*}$ is injective, suppose that $\overline{R(\lambda - a(D)\mbox{*})} = H^{-s}_{c}(I).$ The previous equality implies that $u = 0$, which contradicts the inicial hypothesis, so $\overline{R(\lambda - a(D)\mbox{*})} \neq H^{-s}_{c}(I)$ and therefore  $\sigma_{p}(a(D)) \subset \sigma_{p}(a(D)\mbox{*}) \cup \sigma_{r}(a(D)\mbox{*}).$ 

Joining this result with the first step we get 
$$\sigma_{p}(a(D))\cup \sigma_{r}(a(D)) \subset \sigma_{p}(a(D)\mbox{*})\cup \sigma_{r}(a(D)\mbox{*}),$$ which is exactly $(i)$.

\textbf{Step III: $\sigma_{p}\left(a(D)\mbox{*}\right) \subset \sigma_{p}(a(D)) \cup \sigma_{r}(a(D))$. }

For $\lambda \in \sigma_{p}(a(D)\mbox{*}),$ we have $(\lambda - a(D)\mbox{*})g = 0$ for some $g\not =0$ in  $D[a(D)\mbox{*}]$ and so
$
\left\langle \left(\lambda - a(D)\mbox{*}\right)g, u \right\rangle = 0, \; \forall \; u \in H^{s}_{loc}(I).
$

In particular, it is true that 
$$
\left\langle g, (\lambda - a(D))u \right\rangle = \left\langle (\lambda - a(D)\mbox{*})g, u \right\rangle = 0, \; \forall \; u \in H^{s+m}_{0}(I).
$$
with $g \neq 0$ and, once again, by Theorem \ref{hahn-banach-frechet}, we conclude that $\overline{R\left(\lambda - a(D)\right)} \neq H^{s}_{loc}(I)$ and then $\lambda \in \sigma_{p}(a(D)) \cup \sigma_{r}(a(D)),$ which establishes $(ii)$.

\textbf{Step IV: $\sigma_{r}(a(D)\mbox{*}) \subset \sigma_{p}(a(D)) \cup \sigma_{c}(a(D))$.}

Consider $\lambda \in \Comp$ such that $\lambda - a(D)$ is injective and $\overline{R(\lambda - a(D))} = H^{s}_{loc}(I),$ which means that $\lambda \notin \sigma_{p}(a(D))\cup \sigma_{r}(a(D))$ and therefore, from Step III, follows that $\lambda \notin \sigma_{p}(a(D))\mbox{*},$ i.e., $\lambda -a(D)\mbox{*}$ is injective. Let's show that 
$$
(\lambda - a(D))^{-1}: R(\lambda - a(D)) \subset H^{s}_{loc}(I) \longrightarrow H^{s}_{loc}(I)
$$
is continuous, so $R(\lambda - a(D)\mbox{*}) = H^{-s}_{c}(I),$ i.e., in other words, if $\lambda \in \rho(a(D))$ then $\lambda \notin \sigma_{r}(a(D)\mbox{*})$, i.e., $\sigma_{r}(a(D)\mbox{*}) \subset \sigma(a(D)).$  

First, we shall prove the following equalities
\begin{equation}
\left[(\lambda - a(D))^{-1}\right]\mbox{*}\left(\lambda - a(D)\mbox{*}\right)g \mid_{R(\lambda - a(D))} = g \mid_{R(\lambda - a(D))}, \;g \in D(a(D)\mbox{*}) \label{identidade-a(D)*1}
\end{equation}
and
\begin{equation}
(\lambda - a(D)\mbox{*})\left[(\lambda - a(D))^{-1}\right]\mbox{*}g = g, \;g \in D\left[(\lambda - a(D))^{-1}\right]\mbox{*}. \label{identidade-a(D)*2}
\end{equation}

Indeed, for  $g \in D(a(D) \mbox{*})$ we have 
\begin{equation*}
     \left\langle \left(\lambda - a(D)\mbox{*}\right)g, \left(\lambda - a(D)\right)^{-1}f \right\rangle = \langle g,f \rangle,  \; f \in R(\lambda - a(D))
\end{equation*}
and then $\big[(\lambda - a(D)\mbox{*}) g \big] \circ \left(\lambda - a(D)\right)^{-1}$ is continuous, which proves that  $R(\lambda - a(D)\mbox{*}) \subset D\left[(\lambda - a(D))^{-1}\right]\mbox{*}$ and \eqref{identidade-a(D)*1} holds.

Given $g \in D\left[(\lambda - a(D))^{-1}\right]\mbox{*}$ we have $\left\langle \left[\left(\lambda - a(D)\right)^{-1}\right] \mbox{*} g, \left(\lambda - a(D)\right)u \right\rangle =$
$$
= \left\langle g, \left(\lambda - a(D)\right)^{-1}\left(\lambda - a(D)\right)u \right\rangle = \langle g, u \rangle, \;u \in H^{s+m}_{0}(I),
$$
so
$\left\{\left[\left(\lambda - a(D)\right)^{-1}\right]\mbox{*}g\right\} \circ (\lambda - a(D)): H^{s+m}_{0}(I) \longrightarrow \Comp
$
is continuous considering the topology induced by $H^{s}_{loc}(I).$ Hence $\left[\left(\lambda - a(D)\right)^{-1}\right]\mbox{*}g \in D(a(D) \mbox{*})$ and \eqref{identidade-a(D)*2} holds, i.e.,
$$
g = \left(\lambda - a(D)\mbox{*}\right)\left[\left(\lambda - a(D)\right)^{-1}\right]\mbox{*}g, \;g \in D\left[\left(\lambda - a(D)\right)^{-1}\right]\mbox{*},
$$
so $D\left[\left(\lambda - a(D)\right)^{-1}\right]\mbox{*} \subset R\left(\lambda - a(D)\mbox{*}\right),$ which was the inclusion needed for us to conclude the equality
$
D\left[\left(\lambda - a(D)\right)^{-1}\right]\mbox{*} = R\left(\lambda - a(D)\mbox{*}\right).
$

Now, from the continuity of $\left(\lambda - a(D)\right)^{-1}$, we have
$$
R\left(\lambda - a(D)\mbox{*}\right) = D\left[(\lambda - a(D))^{-1}\right]\mbox{*} = H^{-s}_{c}(I).
$$ 
Thus if $\lambda \in \rho(a(D))$, we conclude that $\lambda \notin \sigma_{r}\left(a(D)\mbox{*}\right),$ which means that $\sigma_{r}\left(a(D)\mbox{*}\right) \subset \sigma(a(D)).$

This fact with Step I implies that $\sigma_{r}\left(a(D)\mbox{*}\right) \subset \sigma_{p}(a(D)) \cup \sigma_{c}(a(D)).$ From Step I we know that $\sigma_{r}\left(a(D)\right) \subset \sigma_{p}(a(D)\mbox{*})$ and $\sigma_{r}\left(a(D)\mbox{*}\right) \cap \sigma_{p}(a(D)\mbox{*})= \emptyset,$ which proves $(iii)$ as we wanted.

\end{proof}

Using the above theorem and an additional hypothesis we are able to calculate $\sigma(a(D))$ and $\sigma(a(D)\mbox{*})$ and give a description for all types of spectrum. 

\begin{theorem}\label{teo-principal}
Under the hypotheses of the last theorem with $a(D)'$ elliptic, $a(D)$ and its adjoint $a(D)\mbox{*}$ both have empty resolvent set and, independently of $s\in \R$, their types of spectrum are classified as follows:
$$\sigma_{p}(a(D)) = \sigma_{p}\left(a(D)\mbox{*}\right) = \emptyset,$$ $$\sigma_{r}(a(D)) =\sigma_{c}\left(a(D)\mbox{*}\right) = \emptyset \mbox{ and }$$ $$\sigma_{c}(a(D)) =\sigma_{r}\left(a(D)\mbox{*}\right) = \Comp.$$

\end{theorem}

\begin{proof}

First of all, lets prove that $\sigma_{p}\left(a(D)\mbox{*}\right) = \emptyset.$ 

If $\left(\lambda-a(D)\mbox{*}\right)g = 0,$ for some $g \in D\left(a(D)\mbox{*}\right),$ then $g(x) = \sum_{j=1}^ {m}C_{j}e^{\beta_{j} x}$ for $C_1, \cdots,C_n \in \Comp$ where 
$\beta_{1},\ldots,\beta_{m} \in \Comp$ are the roots of the polynomial \footnote{If $m=2$ and $\lambda - a\mbox{*}(\xi)=0$ has only one root $\beta_{0}$, as we know from ODE's, the solution is given by $g(x) = C_{1}e^{\beta_{0} x} + C_{2}x e^{\beta_{0} x}.$ For the general case we proceed in an analogous way for each non simple root.}
$\lambda - \sum_{k=1}^{m} (-2\pi i)^{-k} a_{k}\xi^{k}. $
Since $D(a(D)\mbox{*}) = H^{-s+m}_{c}(I),$ $g$ has compact support which implies that $g \equiv 0$, so $\sigma_{p}(a(D)\mbox{*}) = \emptyset.$

We claim that there exist $u\not=0$ in $ H^{s}_{loc}(I)$, such that
\begin{equation}
\left\langle u, \left(\lambda - a(D)\mbox{*}\right)g \right\rangle = 0, \; \forall g \in D\left(a(D)\mbox{*}\right),
\label{simbolo-lambda-a(D)}
\end{equation}
which will give us that $\overline{R(\lambda - a(D)\mbox{*})} \neq H^{-s}_{c}(I).$

Indeed, note that if $u \in C^{\infty}(I) \subset H^{s}_{loc}(I)$ then
$
\langle u, (\lambda - a(D)\mbox{*})g \rangle = \langle (\lambda - a(D))u, g \rangle, \; g \in D(a(D)\mbox{*}).
$
So, if we chose  $u(x) = e^{\xi_{0} x},$ where $\xi_{0} \in \Comp$ is a root of the polynomial 
$
\lambda - \sum_{k=1}^{m} (2\pi i)^{-k} a_{k}\xi^{k},
$
 then we get that $u \in C^{\infty}(I),$ $u\not =0$ and satisfies $(\lambda - a(D))u = 0.$ For this reason
$$
\langle u, (\lambda - a(D)\mbox{*})g \rangle = \langle (\lambda - a(D))u, g \rangle = 0, \; \forall \; g \in D(a(D)\mbox{*})
$$
and we conclude that $\sigma_{r}(a(D)\mbox{*}) = \Comp.$

Now we show that $\sigma_{p}(a(D)) = \emptyset.$ To do so, note that $(\lambda - a(D)) u = 0$ implies $u(x) = \sum_{j=1}^{m}A_{j}e^{\beta_{j} x}$, for $A_{j} \in \Comp$ and $\beta_{1},\ldots, \beta_{m} \in \Comp$ the roots of \eqref{simbolo-lambda-a(D)}. \footnote{The comment of the previous proof for $\sigma_{p}(a(D)\mbox{*}) = \varnothing$ is valid here.} Nevertheless, in order to $u \in H^{s+m}_{0}(I)$, we need that $u,u',\ldots,u^{(m-1)}$ are equal to zero on $\partial I.$ Since for each $l\in {\mathbb N}$,
$u^l(x) = \sum_{j=1}^{m} A_{j} \beta_{j}^{l}e^{\beta_{j}x},$ 
wrinting $I=(b,c)$, we get the following system of equations
$u(b)  =  \sum_{j=1}^ {m}A_{j}e^{\beta_{j} b} = 0,  \cdots,$
$u^{m-1}(b)  =  \sum_{j=1}^{m}A_{j} \beta_{j}^{m-1} e^{\beta_{j} b}  = 0$
and
$u(c) = \sum_{j=1}^{m} A_{j}e^{\beta_{j} c} = 0,  \cdots,$
$u^{m-1}(c)  =  \sum_{j=1}^{m} A_{j} \beta_{j}^{m-1} e^{\beta_{j} c}  = 0.$

Solving them, on the $A_j$'s variables, we conclude that $u \equiv 0$ and, therefore, $\sigma_{p}(a(D)) = \emptyset.$ 

Finally, from what we have proved here with the inclusions given by the previous theorem, it gives us that $\Comp = \sigma_{r}(a(D)\mbox{*}) \subset \sigma_{p}(a(D)) \cup \sigma_{c}(a(D)) = \sigma_{c}(a(D))$, then $\sigma_{c}(a(D)) = \Comp,$ completing the proof. 
\end{proof}

\subsection{Closure of a Differential Operator on a Fr\'echet Space}

Here we determine the closure of a differential operator with constant coefficients $a(D)$ of order $m \geq 1$ on $H^{s}_{loc}(I).$ That will allow us to obtain a more precise analysis of the spectrum, in the sense that we can track the change of the values $\lambda \in \Comp$ as we close the operators.

First of all, we need to construct a convenient sequence of functions that will be the main tool to make the calculus of the closure (check \cite{brezisaf} to see the inspiring construction).

Let $I=(a,b)$ be an interval. Given a function $f \in H^{s}_{loc}(I)$, $s \in \R,$ consider its null extension
$$
f_{e}(x) = 
\left \{ 
\begin{array}{cc} 
f(x), \text{ if }  x \in I \\ 
0,  \text{ if }  x \in \R \setminus I, 
\end{array}  
\right. 
$$
which is an element of $H^{s}_{loc}(\R \setminus \partial I ).$ 

Now let $(I_{j})_{j \in \N}$, $I_{j} = (a_{j},b_{j})$ a sequence of open bounded intervals with $I = \bigcup_{j \in \N} I_{j},$ $\overline{I_{j}} \subset I_{j+1}$ and $d(I_{j}, \R \setminus I) \geq 2/j.$

Define $g_{j} = \chi_{I_{j}} \cdot f_{e}$\, \footnote{Note that $g_{j}\in L^{p}(\R),$ for each natural $j$.} and $f_{j} = \phi_{j} \star g_{j},$ where $\chi_{I_{j}}$ is the characteristic function of $I_{j},$ $\phi_{j} \in C^{\infty}_{c}(-1/j,1/j)$ with $\phi_{j} \geq 0$ and $\int_{\R}\phi_{j} = 1,$ for every $j \in \N$.
Note that $f_{j} \in C^{\infty}_{c}(I).$

Given $u \in H^{s}_{loc}(I),$ with $s \in \Z, s \geq 0,$ consider $u=f$ in the above construction, then $u_{j} \doteq \phi_{j} \star g_{j}$, where $g_{j} \doteq \chi_{I_{j}}u_{e}.$

It follows that, for each $0 \leq k \leq s$, $u_{j}^{(k)} = \phi_{j} \star g_{j}^{(k)}$ as distributions in $\Dl(\R)$ and, therefore, since $a$ is of order $m$, we need to find the derivatives $g_{j}', g_{j}'', \ldots, g_{j}^{(m)}$ in $\Dl(\R).$ 


Given $\psi \in C^{\infty}_{c}(\R)$ it follows that
$$
    \langle g_{j}', \psi \rangle = - \langle g_{j}, \psi' \rangle = - \displaystyle \int_{\R} u_{e}(x) \psi'(x) dx 
    = - \int_{I_{j}} u(x) \psi'(x) dx. 
$$
Since $u \mid _{I_{j}} \in H^{s}(I_{j}) \hookrightarrow C^{s-1}(\overline{I_{j}}),$ we have 
\begin{equation} \label{equacaohs}
(\psi u)' = \psi' u + \psi u' \text{ in }  C(I_{j}) \Longleftrightarrow \psi' u  =  (\psi u)' - \psi u' \text{ in } C(I_{j}),
\end{equation}
and then $\langle g_{j}', \psi \rangle = -  \int_{I_{j}} [(u(x) \psi(x))' - u'(x) \psi(x) ] dx = $
\begin{eqnarray*}
     &=& - \displaystyle  \int_{I_{j}} [ u'(x) \psi(x) - (u(x)\psi(x))'] dx = \int_{I_{j}} u'(x) \psi(x) dx - [\psi(x) u(x)]|_{\partial I_{j}}= \\ 
   &=& \langle \chi_{I_{j}}\cdot u' , \psi \rangle  - [\psi(x) u(x)]\mid_{\partial I_{j}} = \langle \chi_{I_{j}}\cdot u', \psi \rangle - u(b_{j}) \langle \delta_{b_{j}}, \psi \rangle + u(a_{j}) \langle \delta_{a_{j}}, \psi \rangle,
\end{eqnarray*}
so we conclude that $g_{j}' = \chi_{I_{j}}\cdot u' - u(b_{j})\delta_{b_{j}} + u(a_{j})\delta_{a_{j}}$ in $\Dl(\R).$

Observe that using the translation $(\tau_{h} \psi)(x)=\psi(x-h)$ and the reflection $(r\psi) (x)=\psi(-x)$, we can write $
(\delta_{p} \star \psi)(x) = \psi(x-p),
$
for every $x,p \in \R.$ 

Applying the same argument as above to $\chi_{I_{j}}\cdot u'$ we conclude that 
$$
g_{j}'' = \chi_{I_{j}} \cdot u'' + u'(a_{j})\delta_{a_{j}} - u'(b_{j})\delta_{b_{j}} + u(a_{j})\delta_{a_{j}}' - u(b_{j})\delta_{b_{j}}'.  
$$
Hence $u''_{j} = \phi_{j} \star g''_{j} = \phi_{j} \star \big[u''\mid _{I_{j}} + u'(a_{j})\delta_{a_{j}} - u'(b_{j})\delta_{b_{j}} + u(a_{j})\delta_{a_{j}} - u(b_{j})\delta_{b_{j}} \big] =$
$$
    \phi_{j} \star [\chi_{I_{j}}u''] +  u'(a_{j})[\phi_{j} \star \delta_{a_{j}}] - u'(b_{j}) [\phi_{j} \star \delta_{b_{j}}]  + u(a_{j})[\phi_{j} \star \delta_{a_{j}}'] - u(b_{j})[\phi_{j} \star \delta_{b_{j}}'], 
$$
i.e., $u''_{j} =\phi_{j} \star [\chi_{I_{j}}u''] +  \sum_{l=0}^{1} \big (u^{(l)}(a_{j})\phi_{j}^{(1-l)}(\cdot -a_{j}) - u^{(l)}(b_{j}) \phi_{j}^{(1-l)}(\cdot -{b_{j}}) \big ).$

For the general case we use induction. Suppose that
$$
g_{j}^{(k)} = \chi_{I_{j}} \cdot u^{(k)}  + \sum_{l=0}^{k-1} \left(u^{(l)}(a_{j})\delta_{a_{j}}^{(k-1-l)} - u^{(l)}(b_{j}),\delta_{b_{j}}^{(k-1-l)}\right)
$$
with $k \leq m-1,$ then
$
g_{j}^{(k+1)} = [\chi_{I_{j}}\cdot u^{(k)}]' + \sum_{l=0}^{k-1} \left(u^{(l)}(a_{j})\delta_{a_{j}}^{(k-l)} - u^{(l)}(b_{j})\delta_{b_{j}}^{(k-l)}\right).
$
Denote $h_{j} \doteq \chi_{I_{j}} \cdot u^{(k)} \in H^{s+m-k}(I_{j}),$ the sentence \eqref{equacaohs} is true in $C^{m-k-1}(I_{j})$ and proceeding as we did for the first derivative of $g_{j},$ it follows that
$
h'_{j} = \chi_{I_{j}} \cdot u^{(k+1)} + u^{(k)}(a_{j})\delta_{a_{j}} - u^{(k)}(b_{j})\delta_{b_{j}}.
$
Therefore 
$$
g_{j}^{(k+1)} = \chi_{I_{j}} \cdot u^{(k+1)} + \sum_{l=0}^{k} \left(u^{(l)}(a_{j})\delta_{a_{j}}^{(k-l)} - u^{(l)}(b_{j})\delta_{b_{j}}^{(k-l)}\right)
$$

So $u^{(k)}_{j} = \phi_{j} \star g^{(k)}_{j} = \phi_{j} \star \left[\chi_{I_{j}} \cdot u^{(k)} + \sum_{l=0}^{k-1}\left(u^{(l)}(a_{j})\delta^{(k-1-l)}_{a_{j}} - u^{(l)}(b_{j})\delta^{(k-1-l)}_{b_{j}}\right)\right]$ and, simplifying, we get $u^{(k)}_{j}=$
 $$   =\phi_{j} \star [\chi_{I_{j}}u^{(k)}] + \sum_{l=0}^{k-1} \left\{u^{(l)}(a_{j})\phi_{j}^{(k-1-l)}(\cdot - a_{j}) - u^{(l)}(b_{j}) \phi_{j}^{(k-1-l)}(\cdot - b_{j})\right\}.$$

The following lemma is fundamental for our purposes.

\begin{lemma} \label{lemmaseque}
Given $u \in H^{s}_{loc}(I)$ with $s \in \Z$ and $s \geq 0,$ for each $0 \leq k \leq s,$ the sequences of functions
$$
\left(u^{(l)}(a_{j})\phi_{j}^{(k-l)}(\cdot - a_{j})\right)_{j \in \N}, \;\;
\left(u^{(l)}(b_{j}) \phi_{j}^{(k-l)}(\cdot - b_{j})\right)_{j \in \N},
$$
where $0 \leq l \leq k-1,$
converge to zero in $H^{s}_{loc}(I).$
\end{lemma}

\begin{proof}
We show the lemma for the case $m=2.$ The proof for the other cases is analogous.

For $\phi \in C^{\infty}_{c}(I),$
we have $\phi(\cdot )u'(a_{j})\phi_{j}(\cdot -a_{j}) \in C^{\infty}_{c}(\R).$ We claim  $\phi(\cdot )u'(a_{j})\phi_{j}(\cdot -a_{j})$ converges to $0$ in the topology of $\mathcal{S}(\R).$ Since $\mathcal{S}(\R) \hookrightarrow H^{s}(\R),$ it follows that the sequence converges to $0$ in $H^{s}(\R)$ and, therefore, $u'(a_{j})\phi_{j}(x-a_{j})$ converges to $0$ in $H^{s}_{loc}(\R).$

Note that 
$
[\phi(x)u'(a_{j})\phi_{j}(x-a_{j})]^{(k)} = u'(a_{j})\sum_{l=0}^{k} \binom{k}{l}\phi^{(k-l)}(x)\phi_{j}^{(l)}(x-a_{j})
$
for $k \in \N,$ so it remains to prove that
$
\sup_{x \in \R} |x|^{N} |\phi^{(k-l)}(x)\phi_{j}^{(l)}(x-a_{j})| \to 0,$ as $j \to \infty,
$
for $N \in \N.$ 

In order to do it, observe that $\supp \phi \subset I=(a,b)$ is compact, so $d(\supp \phi, a) > 0.$ On the other hand, $\supp \phi_{j}(\cdot - a_{j}) \subset B_{1/j}(a_{j})$ so there is a $j_{0} \in \N$ such that $B_{1/j}(a_{j}) \cap \supp \phi = \emptyset.$

Hence
$
\sup_{x \in \R} |x|^{N} \left|\phi^{(k-l)}(x)\phi_{j}^{(l)}(x-a_{j})\right| = 0,$ for $j \geq j_{0}
$
and
$$
\sup_{x \in \R}|x|^{N}\left|[\phi(x)u'(a_{j})\phi_{j}(x-a_{j})]^{(k)}\right| \leq \left|u'(a_{j})\right|\sum_{l=0}^{k}\sup_{x \in \R} |x|^{N} \left|\phi^{(k-l)}(x)\phi_{j}^{(l)}(x-a_{j})\right| 
$$
which is equal to $0$
for $j \geq j_{0}$ and the convergence holds.

The proof for the other sequences can be done in an analogous way.
\end{proof}

\begin{lemma} \label{lemma-hsloc}
If $h \in H^{s}_{loc}(I),$ with $s \in \Z$ and $s \geq 0,$ then $h_{j} \doteq \phi_{j} \star (\chi_{I_{j}} h_{e}) \in C^{\infty}_{c}(I), j \in \N,$ converges to $h$ in $H^{s}_{loc}(I).$
\end{lemma}

\begin{proof}

First of all, we shall prove this lemma holds for $s=0$, i.e., if $h \in L^{2}_{loc}(I)$, then $h_{j}= \phi_{j} \star (\chi_{I_{j}} h_{e})$ converges to $h$ in $L^{2}_{loc}(I).$

Indeed, given $h \in L^{2}_{loc}(I)$ and $\varphi_{l}$ test function from a seminorm of $L^{2}_{loc}(I)$ it holds 
$$
\|\varphi_{l}h - \varphi h_{j}\|_{L^{2}(\R)} = \|\varphi_{l} h_{e} - \varphi_{l} h_{j}\|_{L^{2}(\R)} \leq C \|h_{e} - h_{j}\|_{L^{2}(I_{l+1})},
$$
where $C>0$ is a constant which depends only on $\varphi_{l}.$ 

Furthermore, for $j$ sufficiently large that $I_{l+1} + B_{1/j} \subset I_{l+2} \subset I_{j}$, we may write
$ h_{e}(x) - h_{j}(x) = \int_{B_{1/j}} \phi_{j}(y) (h_{e}(x)-h_{e}(x-y)) dy.$

However, note that $\phi_{j}(y)(h_{e}(\cdot)-h_{e}(\cdot-y)) \in L^{2}(I_{l+1})$ for every $y \in B_{1/j}$ and $\phi_{j}(\cdot)\big (h_{e}(x)-h_{e}(x-\cdot)\big) \in L^{2}(I_{1/j})$ for every $x \in I_{l+1},$ then by the Minkowski Inequality for Integrals it follows
$\|h_{e} - h_{j}\|_{L^{2}(I_{l+1})} \leq$
$$
\int_{B_{1/j}} |\phi_{j}(y)|\|h_{e}(\cdot) - h_{e}(\cdot - y)\|_{L^{2}(I_{l+1})} dy,$$ i.e.,
$\|h_{e} - h_{j}\|_{L^{2}(I_{l+1})} \leq \int_{\R} \phi_{j}(y)\chi_{B_{1/j}}(y)\|h_{e}(\cdot) - h_{e}(\cdot - y)\|_{L^{2}(I_{l+1})} dy.
$

Moreover, 
$
\|h_{e}(\cdot)\|_{L^{2}(I_{l+1})} = \|h\|_{L^{2}(I_{l+1})} \leq \|h\|_{L^{2}(I_{l+2})}
$
and
$$
\|h_{e}(\cdot - y)\|_{L^{2}(I_{l+1})} = 
\left(\int_{I_{l+1}+B_{1/j}} |h_{e}(z)|^{2} dz\right)^{1/2} \leq \left(\int_{I_{l+2}} |h_{e}(z)|^{2} dz\right)^{1/2},
$$
i.e., $\|h_{e}(\cdot) - h_{e}(\cdot - y)\|_{L^{2}(I_{l+1})} \leq 2 \|h\|_{L^{2}(I_{l+2})}$, with $\phi_{j}(y)\chi_{B_{1/j}}(y)\|h_{e}(\cdot) - h_{e}(\cdot - y)\|_{L^{2}(I_{l+1})}$ converging to zero when $j \to \infty$ a.e. $y \in \R$.

It is also true that
$
\phi_{j}(y)\chi_{B_{1/j}}(y)\|h_{e}(\cdot) - h_{e}(\cdot - y)\|_{L^{2}(I_{l+1})} \leq \chi_{(-1,1)}(y) 2 \|h_{e}\|_{L^{2}(I_{l+2})}
$ and $\chi_{(-1,1)}(y) 2 \|h_{e}\|_{L^{2}(I_{l+2})} \in L^{1}(\R)$
for every $y \in \R.$

By Dominated Convergence Theorem, we get
$$
\lim_{j \to \infty} \int_{\R} \phi_{j}(y)\chi_{B_{1/j}}(y)\|h_{e}(\cdot) - h_{e}(\cdot - y)\|_{L^{2}(I_{l+1})} dy = 0.
$$
In short, given $\epsilon > 0,$ there is a $j_{0} \in \N$ such that 
$
\|\varphi_{l}h - \varphi h_{j}\|_{L^{2}(\R)} \leq C \|h_{e} - h_{j}\|_{L^{2}(I_{l+1})} \leq C\epsilon,
$
for $j \geq j_{0}$ and the convergence in $L^{2}_{loc}(I)$ follows.

Now, for $h \in H^{k}_{loc}(I)$ with $k \in \N$ we have $\frac{d^{r}h}{dx^{r}} \in L^{2}_{loc}(I)$ for $0 \leq r \leq k$ and, by the first part of this proof, that $\phi_{j} \star \left[\chi_{I_{j}}\left(\frac{d^{r}h}{dx^{r}}\right)_{e}\right]$ converges to $\frac{d^{r}h}{dx^{r}}$ in $L^{2}_{loc}(I).$

Observe that 
$
\frac{d^{r}}{dx^{r}} \left[ \phi_{j} \star (\chi_{I_{j}} h_{e}) \right] = \phi_{j} \star \frac{d^{r}}{dx^{r}} (\chi_{I_{j}} h_{e}) 
$
and
$$
\frac{d^{r}}{dx^{r}} (\chi_{I_{j}} h_{e}) = \chi_{I_{j}} \left(\frac{d^{r}h}{dx^{r}}\right)_{e} + \sum_{l=0}^{r-1} \left(u^{(l)}(a_{j})\delta_{a_{j}}^{(r-1-l)} - u^{(l)}(b_{j})\delta_{b_{j}}^{(r-1-l)}\right).
$$
By the previous lemma, the sum $\sum_{l=0}^{r-1} \left(u^{(l)}(a_{j})\delta_{a_{j}}^{(r-1-l)} - u^{(l)}(b_{j})\delta_{b_{j}}^{(r-1-l)}\right)
$ converges to $0$ in $L^{2}_{loc}(I).$

Henceforth $\frac{d^{r}}{dx^{r}} \left[ \phi_{j} \star (\chi_{I_{j}} h_{e}) \right] =$
$$
= \phi_{j} \star \frac{d^{r}}{dx^{r}} (\chi_{I_{j}} h_{e}) =  \phi_{j} \star \left[ \chi_{I_{j}} \left(\frac{d^{r}h}{dx^{r}}\right)_{e} + \sum_{l=0}^{r-1} \left(u^{(l)}(a_{j})\delta_{a_{j}}^{(r-1-l)} - u^{(l)}(b_{j})\delta_{b_{j}}^{(r-1-l)}\right) \right]
$$
and then, for each $1 \leq r \leq k$, it is true that
$
L^{2}_{loc}(I) - \lim_{j \to \infty} \frac{d^{r}}{dx^{r}} \left[ \phi_{j} \star (\chi_{I_{j}} h_{e}) \right] = \frac{d^{r}h}{dx^{r}},
$
i.e., $h_{j} = \phi_{j} \star (\chi_{I_{j}} h_{e})$ converges to $h$ in $H^{k}_{loc}(I).$

\end{proof}

\begin{theorem} \label{teo-fecho}
If $a(D): H^{s+m}_{0}(I) \subset H^{s}_{loc}(I) \to H^{s}_{loc}(I)$ is an elliptic differential operator, with constant coefficients, given by $\sum_{j=1}^{m} (-2\pi i)^ka_{k} u^{(k)}$ where $s \in \Z_{+}$, then its closure is given by $\overline{a(D)}: H^{s+m}_{loc}(I) \subset H^{s}_{loc}(I) \longrightarrow H^{s}_{loc}(I)$ with
$   \overline{a(D)}(u) = \displaystyle \sum_{j=1}^{m} (-2\pi i)^ka_{k} u^{(k)}.
$
\end{theorem}

\begin{proof}
Let 
$
\overline{a(D)}: D\left[\overline{a(D)}\right] \subset H^{s}_{loc}(I) \longrightarrow H^{s}_{loc}(I)
$
be the closure of 
$
a(D): H^{s+m}_{0}(I) \subset H^{s}_{loc}(I) \longrightarrow H^{s}_{loc}(I),
$
where $D\left[\overline{a(D)}\right] =$
$\Big\{ u \in H^{s}_{loc}(I); \; \exists \; (u_{j})_{j \in \N} \subset H^{s+m}_{0}(I) 
\mbox{ and } $ $f \in H^{s}_{loc}(I) 
\mbox{ s.t. } u_{n} \xrightarrow{H^{s}_{loc}} u \mbox{ and } a(D)u_{j} \xrightarrow{H^{s}_{loc}} f \Big\}.$

By the definition of $D\left[\overline{a(D)}\right],$ it immediately follows that $D\left[\overline{a(D)}\right] \subset H^{s+m}_{loc}(I)$, since every $u \in D\left[\overline{a(D)}\right]$ is limit $H^{s}_{loc}$ of a sequence of functions of $H^{s+m}_{0}.$  Furthermore, $u_{n} \xrightarrow{H^{s}_{loc}} u$ and $a(D)u_{n} \xrightarrow{H^{s}_{loc}} f$ imply that $a(D)u = f$ in $\Dl(I)$ and, since $a(D)$ is elliptic, $u \in H^{s+m}_{loc}(I).$

On the other hand, for $u \in H^{s+m}_{loc}(I),$ we have that $f \doteq a(D)u \in H^{s}_{loc}(I)$. Let $u_{n} = \varphi_{n} \star (\chi_{I_{n}} u_{e}),$ $n \in \N,$ then by what we have done above the theorem holds.
\end{proof}




\subsection{Spectrum of the Laplace operator on a Fr\'echet Space}

In this section we apply the results obtained in the previous section to the Laplacian operator. The main characteristic that allow us to apply these results is the fact that both the laplacian and its adjoint are elliptic operators. \footnote{The Laplacian is a self-adjoint operator in the context of pseudodifferential operators.}

First of all, since $\Delta u =u''$, from what we have discussed above, it follows that $\Delta^{*}g=g''=\Delta g$. Furthermore, since $D(\Delta)=H^{2}_0(I)$, we have $D(\Delta^{*})=H^{2}_c(I)$ (here we use $s=0$).

On the other hand, both symbols of $\Delta$ and $\Delta\mbox{*}$ are given by $a(\xi) = -4\pi^{2} \xi^{2}.$ Take $C = 4 \pi^{2}$ then 
$
|a(\xi)| = 4 \pi^{2} \xi^{2} \geq C |\xi|^{2}
$
and, by Definition \ref{def-elipt}, it follows that $\Delta$ and $\Delta\mbox{*}$ are elliptic.

\begin{corollary} The Laplace operator, seen as a  pseudodiferencial operator 
$
\Delta: H^{2}_{0}(0,\pi) \subset L^{2}_{loc}(0,\pi) \longrightarrow L^{2}_{loc}(0,\pi)
$
and its adjoint
$
\Delta^{\ast}: H^{2}_{c}(0,\pi) \subset L^{2}_{c}(0,\pi) \longrightarrow L^{2}_{c}(0,\pi),
$
both have resolvent set empty and their spectra are classified as follows:
$\sigma_{p}(\Delta) = \sigma_{p}(\Delta\mbox{*}) = \emptyset,$
$\sigma_{r}(\Delta) =\sigma_{c}(\Delta\mbox{*}) = \emptyset, \nonumber$ and $\sigma_{c}(\Delta)=\sigma_{r}(\Delta\mbox{*}) = \Comp.$    

\end{corollary}

\begin{proof}
This corollary follows immediately from the ellipticity of $\Delta$ and $\Delta\mbox{*}$ and Theorem \ref{teo-principal}.

\end{proof}

Finally, using the results obtained for  $\Delta$ with Theorem \ref{teo-fecho} its possible to obtain a better analysis for its spectrum.

By Theorem \ref{teo-fecho}, with $s=0$, it follows that
$D\left[\overline{\Delta}\right] = H^{2}_{loc}(I)$ and
$$
\overline{\Delta}: H^{2}_{loc}(I) \subset L^{2}_{loc}(I) \longrightarrow L^{2}_{loc}(I)
$$
is given by $\overline{\Delta}u = u''$, for $u \in H^{2}_{loc}(I).$

Denote by $\Delta_{L^{2}}$ the Laplacian defined on the domain
$$
\Delta_{L^{2}}: H^{1}_{0}(I) \cap H^{2}(I) \subset L^{2}_{loc}(I) \longrightarrow L^{2}_{loc}(I).
$$
Since $H^{2}_{0}(I) \subset H^{1}_{0}(I) \cap H^{2}(I),$ we have $\overline{\Delta} = \overline{\Delta_{L^{2}}}.$ Moreover, its point spectrum, $\sigma_{p}(\Delta_{L^{2}}),$ is the same as when we consider the topology of $L^{2}(I),$ i.e., $\sigma_{p}(\Delta_{L^{2}}) = \left\{-\frac{\pi^2n^{2}}{l(I)^2}: n \in \N \right\}$, where $l(I)$ is the length of $I.$

It remains to calculate $\sigma_{p}(\overline{\Delta}).$ If $\lambda \in {\mathbb C}$ and $u \in H^{2}_{loc}(I)$, $u\not =0$, are such that $u'' = \lambda u$, then $u(x) = C_{1}e^{\beta_{1}x} + C_{2}e^{\beta_{2}x}$, for some $C_{1},C_{2} \in \Comp$ and $\beta_{1}$ and $\beta_{2}$ the roots of $\lambda + \xi^{2},$ $\xi \in \R$. 
Hence, every $\lambda \in {\mathbb C}$ belongs to $\sigma_{p}(\overline{\Delta}).$

By Theorem \ref{fechadofrechet} we have $\sigma(\Delta) = \sigma(\Delta_{L^{2}}) = \sigma(\overline{\Delta}) = \Comp$ and the following table shows the results obtained for the Laplacian:


\begin{table}[ht]
\caption{}
\renewcommand\arraystretch{1.5}
\noindent\[ 
\begin{array}{|c|c|c|c|}
\hline
 & \Delta & \Delta_{L^{2}} & \overline{\Delta}\\
\hline \sigma_{p} & \emptyset   & \left\{-\frac{\pi^2n^{2}}{l(I)^2}: n \in \N \right\}   & \Comp \\
\hline \sigma_{r} & \emptyset   & \emptyset   & \emptyset\\
\hline \sigma_{c} & \Comp   & \Comp \setminus \left\{-\frac{\pi^2n^{2}}{l(I)^2}: n \in \N \right\}  & \emptyset\\
\hline
\end{array} 
\]
\end{table}

{\bf Acknowledgments.} This study was financed in part by the Coordena\c c\~ao de Aperfei\c coamento de Pessoal de N\'ivel Superior – Brasil (CAPES) – Finance Code 001.

\bibliographystyle{amsplain}

\end{document}